\documentclass[12pt]{amsart}
\usepackage{lipsum}

\usepackage[margin=0.8in]{geometry}
\usepackage[scr=rsfs]{mathalpha}
\usepackage{xspace}
\usepackage{amsfonts}
\usepackage{amsthm}
\usepackage{amssymb}
\usepackage{times} 
\usepackage{graphicx}
\usepackage{hyperref}
\usepackage{tikz}
\usepackage{comment}
\usetikzlibrary{arrows}
\usepackage[all]{xy}
\usepackage{tikz-cd}
\usepackage{enumerate}
\usepackage{mathbbol}

\usepackage{mathtools}

\date{}

\numberwithin{equation}{section}
\newtheorem{thm}{Theorem}[section]
\newtheorem{prop}[thm]{Proposition}

\tikzset{node distance=5cm, auto}

\newtheorem{lemma}[thm]{Lemma}

\newtheorem{defn}[thm]{Definition}
\newtheorem{cor}[thm]{Corollary}
\newtheorem{remark}[thm]{Remark}

\newtheorem{question}[thm]{Question}
\makeatletter
\newcommand{\etale}{\'etal\@ifstar{\'e}{e\xspace}}
\makeatother
\usepackage[OT2,T1]{fontenc}
\DeclareSymbolFont{cyrletters}{OT2}{wncyr}{m}{n}
\DeclareMathSymbol{\Sha}{\mathalpha}{cyrletters}{"58}

\begin{document}

\title{Artin-Schreier-Witt lifts  of purely inseparable extensions}
\author{Srinivasan Srimathy}
\email{srimathy@math.tifr.res.in}
\address{School of Mathematics\\
Tata Institute of Fundamental Research\\
Mumbai\\
India}

%
\subjclass[2020]{Primary 11S15, 14G17, 12F15, 12J20}
\keywords{Artin-Schreier-Witt extensions, imperfect fields, discrete valued fields, ramification, lifting problem, Witt vectors.}

\begin{abstract}
Given a discrete valued field $K$ of positive characteristic, we  study the cyclic lifting problem  of purely inseparable extensions of the residue field.   We prove  that unlike the mixed characteristic case, cyclic lifts of any finite purely inseparable modular extension   exist and show how to explicitly construct them. Moreover,  given such a residual extension, we prove the existence of Artin-Schreier-Witt lifts of any finite degree.  More generally, we show that not only we can construct an  arbitrary large  degree Artin-Schreier-Witt lifts, but we can also construct the extension in such a way that the  intermediate Artin-Schreier extensions are of any ramification type that is admissible by the residue. To show this we define the notion of \emph{gene} and explicitly show how to construct $\mathcal{G}$-weaves and $\mathcal{G}$-cyclic extensions where $\mathcal{G}$ is an arbitrary gene over $K$.  As a consequence, this yields an affirmative answer to a question in \cite{ramification_survey} as well as implies that  there is no cap on the wild ramification index unlike the mixed characteristic case.  We also show  some interesting applications  such as constructing cyclic lifts of fields that are isomorphic to the tensor product of a purely inseparable modular extension  and an Artin-Schreier-Witt extension. Finally, we prove a  structure  theorem  of Artin-Schreier-Witt extensions over a discrete valued fields  which restricts the ramification type of intermediate field extensions complementing the above results. 
\end{abstract}

\maketitle


\section{Introduction}\label{sec:intro}
Let $K$ be  discrete valued field with residue $k$  of characteristic $p$ that is not necessarily perfect. For now,  we do not impose any restriction on the characteristic of $K$. Classically, we know that any finite degree separable extension $l/k$    can be lifted to an (unramified) extension $L/K$ of the same degree. Moreover, this extension is unique if $K$ is complete.  Now suppose $l/k$ is a finite purely inseparable extension.  One can certainly lift this to  $L/K$  that is inseparable.  One can also easily construct a separable $L/K$ whose residue is $l/k$ as follows.  Let $[l:k] = p^n$. We know that  there exists a sequence of fields $k = l_0 \subset l_1 \subset \cdots \subset l_{n-1} \subset l_n = l$ such that $l_{i+1} = l_i(\alpha_i^{1/p})$, for some $\alpha_i \in l_i$ (\cite[\href{https://stacks.math.columbia.edu/tag/09HD}{Tag 09HD}, Lemma 9.14.5]{stacks-project}). Let $t$ be a uniformizer of $K$. Then the extension $L_{i+1}$ generated over $L_i$ by $X^p - tX - \tilde{\alpha_i}$ (where $\tilde{\alpha_i}$ is a lift of $\alpha_i$ in $L_i$) is separable with residue $l_{i+1}$. So we get  a  tower $K = L_0 \subset L_1 \subset \cdots \subset L_{n-1} \subset L_n = L$ of separable  extensions yielding a separable $L/K$  whose residue is $l/k$.  \\
\indent Note however that  the above construction need not yield a Galois $L/K$ (the problem being that tower of normal extensions need not be normal). So we may ask if there exists a discrete valued Galois  extension whose residue is $l/k$. A more interesting and difficult question would be to ask whether there exist discrete valued cyclic extensions whose residue is $l/k$.    The answer differs drastically depending on the characteristic of $K$. In the mixed characteristic case, i.e., $char~K =0$, by the results of  Hyodo and Miki (\cite[Proposition (0-4)]{hyodo}, \cite{miki}), given any cyclic extension of $K$,  the degree of inseparability of the residue extension is bounded by a number depending on the absolute ramification index of $K$. In particular, there need not exist cyclic lifts of $l/k$ if the degree $[l:k]$ is arbitrary. It might be interesting to also note that by a result of  Kurihara, there is also a cap on the  wild ramification index if $K$ is of certain type (\cite[Theorem 3.1]{kurihara}). \\
\indent  The goal of this paper is to study the above cyclic lifting problem in the equicharacteristic case. Recall from \cite{sweedler} that  $l/k$ is said to be \emph{modular} if $l$ is isomorphic to the tensor product of primitively generated (a.k.a simple) extensions of $k$. We prove the following as a special case of a more general theorem (Theorem \ref{thm:main}).
\begin{thm} \label{thm:mainintro}
    Let $K$ be a discrete valued field of characteristic $p$ with residue $k$. Suppose $l/k$  is a finite purely inseparable modular extension. Then there exist a cyclic extension  $L/K$  of degree $[l:k]$ with residue $l/k$ and one can explicitly construct it. 
\end{thm}
    As explained above, this is quite contrasting to the mixed characteristic case.  In fact, we prove that given a purely inseparable modular $l/k$, for any $n$ such that $p^n\geq [l:k]$, there exists a  cyclic extension  $L/K$  of degree $p^n$ (a.k.a \emph{Artin-Schreier-Witt extensions}), with residue $l/k$. Existence of such lifts follows from a  more general construction built using the notion of a \emph{gene} (Definition \ref{defn:gene}). We will briefly provide a sketch of ideas involved in the proof without any technical details. Roughly speaking, a gene, denoted by $\mathcal{G}$ is word consisting of a finite sequence of letters each of taking values  from  a  subset of the set  $\{W, F_a, a\in \mathcal{O}_K\}$. Here $W$ stands for \emph{wild} and $F_a$ stands for \emph{ferocious of type $a$} (i.e., the corresponding residue extension is obtained by adjoining  $\overline{a}^{1/p^r}$ for some $r$).   Given a gene $\mathcal{G} = g_1g_2\cdots g_n$ of length $n$, the goal is to construct a \emph{$\mathcal{G}$-cyclic} extension $L/K$. By definition, it is a cyclic extension of degree $p^n$ such that the unique cyclic subextension $L_i/K$ of degree $p^i$ has the property that $L_i/L_{i-1}$ is of the type $g_i$ (Definition \ref{defn:gcyclic}). This is achieved using a result of Albert (Theorem \ref{thm:albert}) and by inductively constructing a cyclic extension  generated by specific generators that  satisfy an equation of a particular form dictated by $\mathcal{G}$. We call the resulting extension as a \emph{$\mathcal{G}$-weave} (Definition \ref{defn:gweave}).  We then prove that $\mathcal{G}$-weaves are $\mathcal{G}$-cyclic yielding the main theorem (Theorem \ref{thm:main}).\\
\indent Naively speaking, the essence of the above narrative is that  not only we can construct an  arbitrary large  degree Artin-Schreier-Witt lift  of  $l/k$, but we can also construct the extension in such a way that the tower of  intermediate Artin-Schreier extensions are of any  "ramification type that is admissible by $l/k$". More precisely,  we explicitly construct a $\mathcal{G}$-cyclic lift where $\mathcal{G}$ is any \emph{$l/k$-admissible gene} (Definition \ref{defn:admissible}, Corollary \ref{cor:main}).  This, in particular gives an affirmative answer to Question 11.1.3 in \cite{ramification_survey} (Corollary \ref{cor:survey}) and also implies that there is no cap on the wild ramification index unlike the mixed characteristic case (Corollary \ref{cor:wild}).    However, we show that  this lift is not unique even if $[L:K]=[l:k]$ and $K$ is complete unlike the case of separable residue extensions. In fact, we show that, given any finite purely inseparable modular $l/k$  and any $l/k$-admissible gene $\mathcal{G}$ of length $[l:k]$ over $K$, one can construct infinitely many non-isomorphic $\mathcal{G}$-cyclic extensions (\S\ref{sec:unique}).  Next in \S \ref{sec:tower}, we  prove  that given any gene $\mathcal{G}$, one can construct   $\mathcal{G}$-cyclic extensions over any given unramified or  split (i.e., the prime of base field splits completely into distinct primes; see Definition \ref{defn:split}) Artin-Schreier-Witt extensions such that the resulting tower is cyclic (Theorem \ref{thm:splitunrammain}). As a corollary we show the existence of cyclic lifts of tensor product of Artin-Schreier-Witt and purely inseparable modular extensions (Corollary \ref{cor:unrammain}) of the residue field. Finally, we prove a  structure  theorem  of Artin-Schreier-Witt extensions over a discrete valued fields in Appendix \ref{app} which restricts the ramification type of intermediate field extensions complementing the above results.

\section{Notations and Preliminaries}\label{sec:notn}
For a discrete valued field $K$, the valuation ring is denoted by $\mathcal{O}_K$, the maximal ideal by $\mathfrak{m}_K$ and the residue field by the corresponding lower case alphabet $k$.   The value group of $K$ is denoted by $\Gamma_K$.  For an element $a \in \mathcal{O}_K$, the symbol $\overline{a} \in k$ denotes its residue. The symbol $\mathcal{P}$ denotes the Artin-Schreier operator, $\mathcal{P}(x) = x^p-x$. The field with $p$ elements is denoted by $\mathbb{F}_p$. The symbol $\mathbb{N}$ denotes the set of natural numbers, i.e., the positive integers   and symbol $\mathbb{N}_{\leq m}$ denotes the set of  natural numbers lesser than or equal to $m$. The supremum of a set $S$ of numbers is denoted by $\sup S$. \\
\indent  Cyclic extensions of degree  $p$, over fields of characteristic $p$  are called as \emph{Artin-Schreier extensions}.  Generalizing this,  cyclic extensions of degree $p^n, n\geq 1$ over fields of characteristic $p$ are called as \emph{Artin-Schreier-Witt extensions} in the literature (\cite{lara}). \\
\indent For a discrete valued cyclic extension $L/K$  of degree $p^n$,   $L_i/K$ denotes  the unique cyclic subextension of degree $p^i$  over $K$ and $l_i/k$ denotes the corresponding residue field extension (Definition  \ref{defn:cyclicsub}).

\begin{defn}\normalfont
   Let $L/K$ be finite Galois extension of discrete valued fields of degree $p^n$.  Let $[l:k]_{sep}$ and  $[l:k]_{insep}$ denote respectively the separable and the inseparable degree of the residual extension $l/k$.  Then $L/K$ is said to be 
    \begin{itemize}
        \item \textit{wild and totally ramified}, if  $[\Gamma_L: \Gamma_K] = p^n$
        \item \textit{ferociously ramified} if   $[l:k]_{insep} = p^n$   \item \textit{completely ramified}, if $[\Gamma_L: \Gamma_K] = p^r$, $[l:k]_{insep} = p^s$ and $r+s = n$
        \item \textit{unramified} if $[l:k]_{sep} = p^n$
    \end{itemize}
\end{defn}
\begin{remark} \label{rmk:dvr} \normalfont
    Let $L/K$ be  Galois extension that is either completely ramified or unramified and let $B$ be the integral closure of $\mathcal{O}_K$ in $L$. Then there is a unique prime ideal in $B$ above $\mathfrak{m}_K$. Therefore $B$ is a discrete valuation ring.
\end{remark}

\begin{defn} \normalfont\label{defn:split}
    Let $A$ be a Dedekind domain  with  field of fractions $F$. We say that a finite field  extension $L/F$ is \emph{split} if  every prime ideal  in $A$ splits  into $[L:F]$ distinct primes in  the integral closure of $A$ in $L$. 
\end{defn}

\section{Gene and $\mathcal{G}$-cyclic extensions}\label{sec:gene}
For the rest of the paper, $K$ denotes a discrete valued field of characteristic $p \neq 0$ with residue $k$. Adapting \cite[\S2.2]{ramification_survey}, we make the following definition.
\begin{defn}\normalfont\label{defn:gene}
Let $\mathfrak{f} \subset k$ denote a $p$-independent set over $k^p$ (See \cite[Chapter 2, \S2.7]{fried_fa} to recall the notion of $p$-independence).  Let $\mathfrak{F}$ be a subset of $\mathcal{O}_K$ containing lifts of elements in $\mathfrak{f}$ satisfying the following properties:
\begin{enumerate}[(i)]
    \item For every $a \in \mathfrak{F}$, $\overline{a} \in \mathfrak{f}$
    \item If $a,b \in \mathfrak{F}$ are distinct elements, then $\overline{a} \neq \overline{b}$
\end{enumerate}

A \emph{gene} over $K$ with alphabet set $\mathfrak{F}$, denoted by $\mathcal{G}_\mathfrak{F}$ is a word of finite length $g_1g_2\cdots g_n$  where for every $i$, the letter $g_i$ takes values from from a set 
    \begin{align}\label{eqn:phi}
        \Phi_\mathfrak{F} = \{ W, F_{a} | a \in \mathfrak{F} \}
    \end{align}
\end{defn}
\noindent  
For the sake of notational convenience, we ignore the subscript and denote the gene simply by $\mathcal{G}$. We also write "a gene over $K$" to mean a gene over $K$ with some fixed underlying alphabet set $\mathfrak{F}$ satisfying (i) and (ii) above.

Here  the \textit{length} of $\mathcal{G}$ is defined to be $n$, the number of letters in the word. For a gene $\mathcal{G}$,   let $\mathcal{G}_i$ denote the truncated word $g_1g_2\cdots g_i$.
\begin{remark}\normalfont
In the above definition $W$ stands for "Wild" and $F_{a}$ stands for "Ferocious of type $a$". Since the elements of $\mathfrak{f}$ are $p$-independent over $k^p$, the collection of fields  $\{k[x]/(x^{p^{r_i}}-\overline{a_i}) | a_i \in \mathfrak{F}\}$ are linearly disjoint  over $k$ for any $\{r_i\}$ and therefore 
\begin{align*}
   l := \bigotimes_{a_i \in \mathfrak{F}} k[x]/(x^{p^{r_i}}-\overline{a_i})
\end{align*}
is a field. Hence, $l/k$ is a purely inseparable modular extension.  
\end{remark}

\begin{defn}\normalfont
    Let  $L/L'$ be a degree $p$  extension of cyclic discrete valued fields and let $K\subseteq L'$ be a subfield. Then $L/L'$ is said to be type $W$ if it is wildly ramified i.e., $\Gamma_L = \frac{1}{p} \Gamma_{L'}$ and  of type $F_a, a \in \mathcal{O}_K$, if $L/L'$ is ferociously ramified with  the residual extension $l \simeq l'(\overline{a}^{1/p^i})$ for some $i>0$.
\end{defn}

\begin{defn}\normalfont \label{defn:cyclicsub} \normalfont
For a discrete valued cyclic extension $L/K$  of degree $p^n$,   $L_i/K$ denotes  the unique cyclic subextension of degree $p^i$  over $K$ and $l_i/k$ denotes the corresponding residue field extension. 
\end{defn}

\begin{defn}\normalfont\label{defn:gcyclic}
Let $\mathcal{G}$ be a gene of length $n$ over $K$. A cyclic extension $L/K$ of degree $p^n$  is said to be \emph{$\mathcal{G}$-cyclic} $L_i/L_{i-1}$ is of type $g_i$ where $g_i$ is the $i$-th letter of $\mathcal{G}$.
\end{defn}
\begin{defn} \normalfont\label{defn:admissible}
    Given a purely inseparable modular extension $l/k$, we say that a gene $\mathcal{G}$ over $K$ with alphabet set $\mathfrak{F}$ is \emph{$l/k$-admissible} if 
   \begin{align*}
   l \simeq \bigotimes_{a \in \mathfrak{F}} k[x]/(x^{p^{l_a}}-\overline{a})
\end{align*}
    \end{defn}
where $l_a$ is the number of occurrences of the letter $F_a$  in the word $\mathcal{G}$.
    \begin{remark} \label{rmk:adm}\normalfont 
        If $\mathcal{G}$ is $l/k$-admissible, then the residual extension corresponding to a $\mathcal{G}$-cyclic extension is $l/k$.
    \end{remark}
    
\section{Main Results}
We keep the notations of \S\ref{sec:gene}. The main theorem of this paper is the following:
\begin{thm}\label{thm:main}
    Given an arbitrary gene $\mathcal{G}$ over $K$, there exists a $\mathcal{G}$-cyclic extension and one can explicitly construct it.
\end{thm}

\begin{proof}
    This follows from Theorem \ref{thm:gweave_gcyclic} and Theorem \ref{thm:gweave_exists}. 
\end{proof}
The above theorem together with Remark \ref{rmk:adm} yields:
\begin{cor} \label{cor:main}
    Given a purely inseparable modular extension $l/k$, there exists a $\mathcal{G}$-cyclic extension for any $l/k$-admissible gene $\mathcal{G}$. In particular, Theorem \ref{thm:mainintro} holds.
\end{cor}

In \cite{ramification_survey}, the authors define the notion of \emph{genome} of a cyclic extension $L/K$ of degree $p^n$ as a word $T_1T_2\cdots T_n$ where 
\begin{align*}
    T_i = \begin{cases}
        W \text{~if~} L_i/L_{i-1} \text{~is~wild}\\
        F\text{~if~} L_i/L_{i-1} \text{~is~ferocious}
    \end{cases}
    \end{align*}
and ask the following question:
\begin{question}(\cite[Question 11.1.3]{ramification_survey})\label{question:genome}
    Given a complete discrete valued field $K$ and a word $T=T_1T_2\cdots T_n$ taking values in the alphabet $\{W,F\}$, does there exist a cyclic extension $L/K$ with genome $T$?
\end{question}

As a direct corollary of Theorem \ref{thm:main}, we get:
\begin{cor} \label{cor:survey}
  The answer to  Question \ref{question:genome} is affirmative even without the assumption that $K$ is complete.
\end{cor}
 Theorem \ref{thm:main} also implies that there is no cap on the wild ramification index unlike the mixed characteristic case (\cite[Theorem 3.1]{kurihara}).
\begin{cor}\label{cor:wild}
    There exists totally (wildly) ramified Artin-Schreier-Witt   extensions of arbitrary finite degree over $K$.
\end{cor}
\begin{proof}
    This follows by taking the gene $\mathcal{G}$ to be the word with all letters $W$ and applying Theorem \ref{thm:main}.
\end{proof}

We also show that one can construct $\mathcal{G}$-cyclic extensions over unramified or split Artin-Schreier-Witt  extensions such that the resulting tower is cyclic. As an interesting corollary, we show that we can also construct Artin-Schreier-Witt lifts of inseparable extensions that are tensor products of Artin-Schreier-Witt extensions and purely inseparable  modular extensions. The proofs of the following can be found in $\S \ref{sec:tower}$.

\begin{thm}
    Let $K/F$ be a discrete valued finite Artin-Schreier-Witt  extension that is either  unramified or split.  Then given an  arbitrary gene $\mathcal{G}$ over $F$, one can construct a  $\mathcal{G}$-cyclic extension $L/K$  such that $L/F$ is cyclic .
\end{thm}
\begin{cor}
    Given any finite  extension $l/k$ such that $l \simeq l_1 \otimes_k l_2$ where $l_1/k$ is Artin-Schreier-Witt  and $l_2/k$ is  purely inseparable modular, there exists  a finite discrete valued Artin-Schreier-Witt  extension $L/K$ with residue $l/k$.
\end{cor}

\begin{cor} 
    Suppose $K$ admits a  split Artin-Schreier extension. Then given any $m\geq 0$ and an arbitrary gene $\mathcal{G}$ over $K$, one can construct Artin-Schreier-Witt  extension $L/K$ such that $L_m/K$ is split and $L/L_m$ is $\mathcal{G}$-cyclic (w.r.t any chosen discrete valuation on $L_m$). In particular, given a finite purely inseparable modular  $l/k$, an $l/k$-admissible gene $\mathcal{G}$ over $K$ and any $m \geq 0$, there exists  discrete valued Artin-Schreier-Witt  extensions $L/K$ with residue $l/k$ such that $L_m/K$ is split and $L/L_m$ is $\mathcal{G}$-cyclic.
\end{cor}
All of the above results are derived from Theorem \ref{thm:main}.  To prove  Theorem \ref{thm:main},  we first define \emph{$\mathcal{G}$-weaves} and  show that they are $\mathcal{G}$-cyclic. Next, we show how to explicitly construct $\mathcal{G}$-weaves for any given gene $\mathcal{G}$.  We do this in the subsequent sections.

\section{$\mathcal{G}$-weaves}

\begin{defn} \label{defn:fns} \normalfont
   \begin{enumerate}
       
\item Let $\mathcal{G} =g_1g_2\cdots g_n$ be a gene of length $n$ with alphabet set $\mathfrak{F}$.  Recall the definition of $\Phi_{\mathfrak{F}}$ from (\ref{eqn:phi}). For every $1\leq i \leq n$,  we define the set
    \begin{align*}
        \mathfrak{A}_{\mathcal{G}}(i) = \{ j| 1 \leq j < i \text{~and~} g_j = g_i\}
    \end{align*}
and the map
    \begin{align*}
        \mathfrak{a}_{\mathcal{G}}: \mathbb{N}_{\leq n} \rightarrow \mathbb{N}_{\leq n-1} \cup \Phi_\mathfrak{F}
    \end{align*}
    
    \begin{align*}
    \mathfrak{a}_{\mathcal{G}}(i)  = \begin{cases}
    \sup \mathfrak{A}_{\mathcal{G}}(i)\text{~~if~~} \mathfrak{A}_{\mathcal{G}}(i) \neq \emptyset\\
 g_i \text{~~else}
    \end{cases}
   \end{align*}
  When the gene under consideration is clear, we skip the subscript and just write  $\mathfrak{A}(i)$ and $ \mathfrak{a}(i)$.   The symbol $\mathfrak{a}^j(i)$ denotes the function $\mathfrak{a}$ applied  $j$ times (if it is defined).  By convention, $\mathfrak{a}^0(i) = i$. 
 \item  For every $F_a, W \in \Phi_\mathfrak{F}$, we define the functions $f_a, w$ associated to $\mathcal{G}$ (as before we ignore attaching $\mathcal{G}$ in the notation for convenience) as follows:
\begin{align*}
    f_a, w: \mathbb{N}_{\leq n} &\rightarrow \mathbb{Z}\\
    f_a(i) &=  |\{j \leq i | g_j = F_a\}|\\
    w(i)&= |\{j \leq i | g_j = W\}|
\end{align*}
\end{enumerate}
\end{defn}
\begin{remark}\normalfont
    Intuitively, $\mathfrak{a}(i)$ denotes the antecedent of $i$, i.e., the previous index $j$ such that $g_j = g_i$. The function takes $\mathfrak{a}(i)$ the value $g_i$ exactly when  none of the  previous letters take the value $g_i$ i.e., $g_i$ occurs for the "first time"  in the word $\mathcal{G}$.  Similarly, $f_a(i)$ and $w(i)$ denote the number of times the letters $F_a$ and $W$ appear in $\mathcal{G}_i$ respectively.
\end{remark}
\begin{remark} \normalfont \label{rem:fa}
    With notations in Definition \ref{defn:fns}, suppose $g_i = F_a$ (resp. $g_i = W$) and $f_a(i) \geq 2$ (resp. $w(i) \geq 2)$, then $f_a(\mathfrak{a}(i)) =  f_a(i-1)=f_a(i) -1$ (resp. $w(\mathfrak{a}(i)) = w(i-1)= w(i) -1$).
\end{remark}
 Without loss of generality, let us  assume that  $\Gamma_K = \mathbb{Z}$. 
\begin{lemma}\label{lem:gcyc}
    Let $L/K$ be $\mathcal{G}$-cyclic. Then for any $i$ and $a \in \mathfrak{F}$,   $\overline{a}^{1/p^{f_a(i)+1}} \notin l_i$ and $\Gamma_{L_i} = (1/p^{w(i)})\mathbb{Z}$
        \end{lemma}
\begin{proof}
    This is clear from the definitions.
\end{proof}

\begin{defn}\normalfont\label{defn:gweave}
    Let $\mathcal{G} = g_1g_2\cdots g_n$ be a gene of length $n$ over $K$ and $t$ denote a uniformizer of $\mathcal{O}_K$.  Set 
    \begin{align} \label{eqn:initial}
        z_{g_i} = \begin{cases}
            a \text{~if~} g_i = F_a\\
            \frac{1}{t} \text{~else}
        \end{cases}
    \end{align}
    A \emph{$\mathcal{G}$-weave} is a cyclic extension $L/K$ of degree $p^n$ with a discrete valuation such that for every $1\leq i \leq n$ there exists $z_i \in L_i$ satisfying 
    \begin{align}\label{eqn:gweave}
        z_i^p - \gamma_{1i}  z_i = z_{\mathfrak{a}(i)} + \gamma_{2i}
    \end{align}
    for some $\gamma_{1i}, \gamma_{2i} \in \mathfrak{m}_{L_{i-1}}$.
\end{defn}

Now we show that a $\mathcal{G}$-weave is $\mathcal{G}$-cyclic.  Let $v$ denote the valuation on $L$. We begin with some definitions and lemmas.
\begin{defn}\label{defn:valcompare}\normalfont
 For $f,g \in L$,  we write $f\prec g$ if $v(f)< v(g)$ and  $f\approx g$ if $v(f) = v(g)$.  We write $f\preceq g$ if $f \prec g$ or $f\approx g$. If $S \subseteq L$, we write $f \prec S$ if $f \prec g, \forall g \in S$.
\end{defn}

\begin{remark}\normalfont
    Note that $f \prec 1$ means $v(f)$ is negative and  $f \approx 1$ means $f$ is a unit in $\mathcal{O}_L$. 
\end{remark}

\begin{lemma}\label{lem:gw}
    Let $L/K$ be a $\mathcal{G}$-weave. With notations as in Definition \ref{defn:gweave},  the following hold:
    \begin{enumerate}[(i)]
    \item For every $i$ such that   $g_i = F_a$ for some $a \in \mathfrak{F}$, we have  $z_i \approx 1$ and
    \begin{align*}
        \overline{z_i} = \overline{a}^{1/p^{f_a(i)}}
    \end{align*}
    
    \item For every $i$ such that  $g_i =  W$,  we have $z_i \prec 1$ and 
    \begin{align*}
       v(z_i) \in (1/{p^{w(i)}}\mathbb{Z}) - (1/{p^{w(i)-1}})\mathbb{Z}
    \end{align*}
    \end{enumerate}
\end{lemma}
\begin{proof}
(i)Let $g_i = F_a$.  We prove the claim by induction on $f_a(i)$. 
 Let $f_a(i)=1$.   By definition, $\mathfrak{a}(i) =g_i = F_a $ and  $z_{g_i} = a$. Therefore by (\ref{eqn:gweave})
     \begin{align*}
         z_i^p- \gamma_{1i}z_i = a  + \gamma_{2i}
     \end{align*}
 for some $\gamma_{1i}, \gamma_{2i} \in \mathfrak{m}_{L_{i-1}}$. Clearly, the above equation implies $z_i \approx 1$ and $\overline{z_i} = a^{1/p}$. So the claim is true when $f_a(i) =1$. Suppose the claim is true for every $i$ with $f_a(i) < m$, where $m\geq 2$.  Let $f_a(i) = m$. By Remark \ref{rem:fa}, $f_a(\mathfrak{a}(i)) = m-1$. Now the element $z_i$ satisfies 
    \begin{align}\label{eqn:temp3}
         z_i^p - \gamma_{1i} z_i = z_{\mathfrak{a}(i)} + \gamma_{2i}   \end{align}
where $\gamma_{1i}, \gamma_{2i} \in \mathfrak{m}_{L_{i-1}}$. By induction hypothesis, $z_{\mathfrak{a}(i)} \approx 1$ and  
\begin{align*}
        \overline{z_{\mathfrak{a}(i)}} = \overline{a}^{1/p^{m-1}}
    \end{align*}
    Therefore from (\ref{eqn:temp3}), we conclude that $z_i \approx 1$ and taking residues we get
     \begin{align*}
        \overline{z_i} = \overline{a}^{1/p^{f_a(i)}}
    \end{align*}
    which finishes the proof.
    \\

\noindent (ii) Let $g_i = W$. We prove the claim again   by induction on $w(i)$. Suppose $w(i) =1$. 
 By definition, $\mathfrak{a}(i) =g_i =W$ and  $z_{g_i} = 1/t$ and therefore by (\ref{eqn:gweave})
     \begin{align*}
         z_i^p- \gamma_{1i}z_i = \frac{1}{t}  + \gamma_{2i}
     \end{align*}
     where $\gamma_{1i}, \gamma_{2i} \in \mathfrak{m}_{L_{i-1}}$. It is easy to see from the above equation that  $z_i \prec 1$ and 
     \begin{align*}
         v(z_i) =  -1/p \in (1/p)\mathbb{Z} - \mathbb{Z}
     \end{align*}
    So the claim is true when $w(i)=1$.\\

Suppose the claim is true for  every $i$  with $w(i) < m$ where $m\geq 2$.  Let $w(i) = m$. Now $z_i$ satisfies 
     \begin{align*}
         z_i^p - \gamma_{1i} z_i = z_{\mathfrak{a}(i)} + \gamma_{2i}   \end{align*}
    By induction hypothesis, $z_{\mathfrak{a}(i)}  \prec 1$ and hence from the above equation we get $z_i \prec 1$. Therefore
    \begin{align*}
    v(z_i) &= (1/p) v(z_{\mathfrak{a}(i)}) \\
    & \in (1/p)( (1/p^{w(i)-1}) \mathbb{Z} - (1/p^{w(i)-2}) \mathbb{Z}) \hspace{.5cm} \text{(Remark \ref{rem:fa})}\\
    & \in (1/{p^{w(i)}}\mathbb{Z}) - (1/{p^{w(i)-1}})\mathbb{Z}
    \end{align*}
as claimed.   
\end{proof}

\begin{thm} \label{thm:gweave_gcyclic}
    A $\mathcal{G}$-weave is $\mathcal{G}$-cyclic.
\end{thm}
\begin{proof}
Let $L/K$ be a $\mathcal{G}$-weave. We prove by induction on $i$ that  $L_i/K$ is $\mathcal{G}_i$-cyclic where $\mathcal{G}_i$ is the truncated word of length $i$ (Definition \ref{defn:gene}). The base case is $i=1$. 
Suppose $g_1 = F_a$ for some $a$. In this case, by Lemma \ref{lem:gw} (i), $\overline{z_1} = \overline{a}^{1/p}$ and hence  we conclude that $L_1/K$ is of type $F_a$.
    On the other hand, if $g_1 = W$, by Lemma \ref{lem:gw} (ii),
     $v(z_1) \in  \frac{1}{p}\mathbb{Z} -\mathbb{Z}$ which means that  $L_1/K$ is of type $W$. Either way, we conclude that $L_1/K$ is of type $g_1$ and hence $L_1/K$ is $\mathcal{G}_1$-cyclic.

Now as induction hypothesis, we assume that $L_{i-1}/K$ is $\mathcal{G}_{i-1}$-cyclic where $i\geq 2$.  Consider the extension $L_i/L_{i-1}$. 
As before, we have two possibilities:
\begin{enumerate}[(i)]
 \item   \textit{Suppose $g_i = F_a$ for some $a$:} 
 
    By induction hypothesis, $L_{i-1}/K$ is $\mathcal{G}_{i-1}$-cyclic and since $g_i = F_a$, $ f_a(i-1) = f_a(i)-1 $. Therefore by Lemma \ref{lem:gcyc}, 
    \begin{align*}
   \overline{a}^{1/p^{f_a(i)}} \notin l_{i-1}
\end{align*}
But $z_i \in L_i$ and  by Lemma \ref{lem:gw}(1), $\overline{a}^{1/p^{f_a(i)}}\in l_i$. Therefore $L_i/L_{i-1}$ is of type $g_i = F_a$.

\item \textit{Suppose $g_i =W$:}     By induction hypothesis, $L_{i-1}/K$ is $\mathcal{G}_{i-1}$-cyclic and since $g_i =W$, $w(i-1) = w(i)-1$. Therefore by Lemma  \ref{lem:gcyc}, $\Gamma_{L_{i-1}}=(1/p^{w(i)-1})\mathbb{Z}$. But $z_i \in L_i$ and by Lemma \ref{lem:gw}(2), 
\begin{align*}
       v(z_i) \in (1/{p^{w(i)}}\mathbb{Z}) - (1/{p^{w(i)-1}})\mathbb{Z}
    \end{align*}
Therefore, $L_i/L_{i-1}$ is of type $g_i =W$.
    \end{enumerate}
    \end{proof}
    
\section{Constructing  $\mathcal{G}$-weaves}
Given a gene $\mathcal{G}$, we will now explicitly construct a $\mathcal{G}$-weave. This requires understanding some special class of polynomials as described below.
\begin{defn} \label{defn:mnspecial} \normalfont
Let $L$ be a field of characteristic $p$ and let $m, n$ be non-negative integers. A polynomial $p(X) \in L[X]$ is said to be $\{m,n\}$-special if it is of the form
\begin{align*}
    p(x) = c + \sum_{i=0}^n a_i X^{p^i}
\end{align*}
where $c, a_i \in L^{p^m}$ and $a_i \neq 0$ for some $i$. We denote the set of $\{m,n\}$-special polynomials by $L^{\{m,n\}}[X]$. 
\end{defn}
The following proposition is an easy exercise left to the reader.
\begin{prop}\label{prop:prop}
    \begin{enumerate}[(i)]
        \item  If $m_2\leq m_1$ and $n_1 \leq n_2$, then  $L^{\{m_1,n_1\}}[X] \subseteq L^{\{m_2,n_2\}}[X]$
        \item Let $f \in L^{\{m_1,n_1\}}[X]$ and $g \in L^{\{m_2,n_2\}}[X]$. Then $g \circ f \in L^{\{min(m_1,m_2),n_1 + n_2\}}[X]$
    \end{enumerate}
\end{prop}

To construct $\mathcal{G}$-weaves, we will need the following important result of Albert. As mentioned in \S\ref{sec:notn}, $\mathcal{P}$ denotes the Artin-Schreier operator.
\begin{thm}(\cite[Lemma 7]{albert_cyclic}, \cite[Theorem 4.2.3]{jacobson})\label{thm:albert}
Let $F$ be a field of characteristic $p$ and let  $E/F$ be a cyclic extension of degree $p^e, e \geq 1$ with $Gal(E/F) = <\sigma>$. Then $E$ contains an element $\beta$ such that $Tr_{E/F}(\beta) =1$ and if $\beta$ is such an element, then there exists an $\alpha \in E$ such that 
\begin{align*}
   \mathcal{P}(\beta) = \sigma(\alpha) - \alpha.
\end{align*}
 Then $x^p -x - \alpha$ is irreducible in $E[x]$ and if $\gamma$ is a root of this polynomial, then $E' = E[\gamma]$ is a cyclic field of degree $p^{e+1}$ over $F$ and any such extension can be obtained this way.   
\end{thm}
\begin{defn}\normalfont
    We call an element $\alpha \in E$ as in the above theorem,  an \textit{Albert element} of $E/F$. For the sake of notational convenience, if $E/F$ is the trivial extension (i.e., $E=F$), we call any element of $F$  an Albert element of $E/F$. 
\end{defn}

\begin{lemma}\label{lem:albert}
    If $\alpha$ is an Albert element of $E/F$, so is $\alpha^{p^N}+c$ for any $c \in F$ and $N \in \mathbb{Z}$.
\end{lemma}
\begin{proof}
    Let $\beta \in E$ such that $Tr_{E/F}(\beta) =1$ and 
    $\mathcal{P}(\beta) = \sigma(\alpha) - \alpha$. We have $Tr_{E/F}(\beta^{p^N}) = Tr_{E/F}(\beta)^{p^N} = 1$ and
    \begin{align*}
        \mathcal{P}(\beta^{p^N}) &=\mathcal{P}(\beta)^{p^N}\\
        &= \sigma(\alpha)^{p^N} - \alpha^{p^N}\\
        &=\sigma(\alpha^{p^N} +c) - (\alpha^{p^N} +c)
    \end{align*}
\end{proof}
\begin{remark} \normalfont 
    If $F$ is a discrete valued field with uniformizer $t$, then  it is easy to check that $1/t \notin \mathcal{P}(F)$. In particular, $F$ admits Artin-Schreier extensions i.e., $F$ is not Artin-Schreier closed.   
\end{remark}
We now show that given any gene $\mathcal{G}$ of finite length, one can explicitly construct a $\mathcal{G}$-weave. The proof involves a few auxiliary lemmas from Appendix \ref{sec:auxi}. As before, $t$ denotes a uniformizer of $\mathcal{O}_K$. Given a natural number $m$, let  $\Sigma(m) = \sum_{i=1}^m i$.  
\begin{thm} \label{thm:gweave_exists}
     Let $\mathcal{G}$ be a gene of length $n$ over $K$ and let $N > \Sigma(n)$. Let $u_{g_i} = z_{g_i}$ where $z_{g_i}$ is given by (\ref{eqn:initial}).   Then one can inductively construct a cyclic extension $L/K$ of degree $p^n$  with  the following properties: For every $1 \leq i \leq n$,
     \begin{enumerate}
      \item   $L_i \simeq L_{i-1}[x_i]$  where $x_i$ satisfies
     \begin{align*}
         x_i^p - x_i = u_i + \alpha_{i-1}^{p^N}
     \end{align*}
     where $u_i = (1/t^{p^{l_i}}) u_{\mathfrak{a}(i)}$ for some $ l_i \geq N$ and $\alpha_{i-1}$ is any Albert element of $L_{i-1}/K$.
     \item there exists $z_i \in L_i$  such that $x_i \in L_{i-1}^{\{N-\Sigma(i), i\}}[z_i]$  and which satisfies 
     \begin{align} \label{eqn:temp7}
         z_i^p - \gamma_{1i} z_i = z_{\mathfrak{a}(i)} + \gamma_{2i}
         \end{align}
         for some $\gamma_{1i}, \gamma_{2i} \in \mathfrak{m}_{L_{i-1}}$.
\end{enumerate}
In particular, $L/K$ is a $\mathcal{G}$-weave.
\end{thm}
\begin{proof}
    We prove by induction on $i$. Let $i=1$. In this case, $\mathfrak{a}(1) = g_1$ and  $x_1$ satisfies
    \begin{align*}
        x_1^p - x_1 = \frac{1}{t^{p^{l_1}}}z_{\mathfrak{a}(1)} + \alpha_0^{p^{N}}
    \end{align*}
   We choose $l_1\geq N$ such that $1/t^{p^{l_1}} \prec \alpha_0^{p^N}$. It is clear that $z_1 = t^{p^{l_1-1}}x_1$ satisfies (\ref{eqn:temp7}) for $i=1$ and  moreover we have $x_1 = t^{-p^{l_1-1}}z_1 \in  K^{\{N-1, 1\}}[z_1]$.  Assume that we have constructed  $L_k, \forall k\leq i-1$ satisfying the theorem. Let $\alpha_{i-1}$ be an arbitrary Albert element in $L_{i-1}/K$. Let $L_i \simeq L_{i-1}[x_i]$  where
    \begin{align*}
         x_i^p - x_i = u_i + \alpha_{i-1}^{p^N}
     \end{align*}
     where $u_i =\frac{1}{t^{p^{l_i}}}u_{\mathfrak{a}(i)}$. There are  two possibilities:
     \begin{enumerate}
         \item  Suppose $\mathfrak{a}(i) = g_i$: Then, 
           $u_{g_i} = z_{g_i} = z_{\mathfrak{a}(i)}$ so that 
               \begin{align*}
         x_i^p - x_i = \frac{1}{t^{p^{l_i}}} z_{\mathfrak{a}(i)}+ \alpha_{i-1}^{p^N}
           \end{align*}
           We choose $l_i>N$ such that $\frac{1}{t^{p^{l_i}}} \prec  \alpha_{i-1}^{p^N}$. Now it is easy to see that $z_i = t^{p^{l_i-1}}x_i$   satisfies (\ref{eqn:temp7}) and the claim. 
      \item Suppose $\mathfrak{a}(i) \neq g_i$: In this case,  $\mathfrak{a}(i) = m \in \mathbb{Z}$ for some $m<i$. Now the field $L_i/L_{i-1}$ is generated by 
      \begin{align*}
          x_i^p - x_i &= \frac{1}{t^{p^{l_i}}}u_{\mathfrak{a}(i)} + \alpha_{i-1}^{p^N} \nonumber\\
          &= \frac{1}{t^{p^{l_i}}}(x_{\mathfrak{a}(i)}^p - x_{\mathfrak{a}(i)}- \alpha_{\mathfrak{a}(i)-1}^{p^N} )+\alpha_{i-1}^{p^N} 
        \end{align*}
By the induction hypothesis, $x_{\mathfrak{a}(i)} \in L_{\mathfrak{a}(i)-1}^{\{N-\Sigma(\mathfrak{a}(i)), \mathfrak{a}(i)\}}[z_\mathfrak{a}(i)] \subseteq L_{i-1}^{\{N-\Sigma(i-1), i-1\}}[z_\mathfrak{a}(i)]$ (Proposition \ref{prop:prop} (i)). Now we may choose large enough $l_i$ and  invoke Lemma \ref{lem:simplify2} which yields  $g \in L_{i-1}^{\{N-\Sigma(i), i-1\}}[z_\mathfrak{a}(i)]$ such that  $y_i$ defined by 

$$y_i := x_i - g(z_\mathfrak{a}(i))$$
satisfies
        \begin{align}\label{eqn:temp9}
            y_i^p -y_i = hz_{\mathfrak{a}(i)} + \eta
        \end{align}
   for some $h \in L_{i-1}^{p^{N-\Sigma(i) +1}}, \eta \in L_{i-1}^{p^{N-\Sigma(i)}}$. Moreover, by the same lemma since  $l_i$ is large enough  $h \prec \{1, \eta\}$. Let $z_i = y_i/h^{1/p}$. Then  (\ref{eqn:temp9}) reduces to
    \begin{align}\label{eqn:temp10}
        z_i^p - \gamma_{1i} z_i = z_{\mathfrak{a}(i)} + \gamma_{2i}
    \end{align}
       for some $\gamma_{1i},\gamma_{2i} \in  L_{i-1}^{p^{N-\Sigma(i)}}\cap \mathfrak{m}_{L_{i-1}}$. Clearly from (\ref{eqn:temp10}), we have $z_{\mathfrak{a}(i)} \in L_{i-1}^{\{N-\Sigma(i),1\}}[z_i]$. We have 
       \begin{align} \label{eqn:temp11}
       x_i = y_i + g(z_\mathfrak{a}(i)) = h^{1/p}z_i + g(z_\mathfrak{a}(i)) 
    \end{align}
       Note that $g \circ z_\mathfrak{a}(i) \in L_{i-1}^{\{N-\Sigma(i),i\}}[z_i]$ by the composition law (Proposition \ref{prop:prop} (ii)). Moreover by (\ref{eqn:temp10}), it is easy to see that the degree of $g \circ z_\mathfrak{a}(i)$  as a polynomial in $z_i$ is at least $p$. Therefore from (\ref{eqn:temp11}), we get $x \in L_{i-1}^{\{N-\Sigma(i),i\}}[z_i]$ as required.
       \end{enumerate}
       It is clear from the  above construction that $L/K$ is cyclic (Theorem \ref{thm:albert}, Lemma \ref{lem:albert}). This together with (\ref{eqn:temp7}) concudes that $L/K$ is a $\mathcal{G}$-weave.
    \end{proof}     
We conclude this section by  pointing out again that  the main theorem (Theorem \ref{thm:main}) follows from Theorem \ref{thm:gweave_gcyclic} and Theorem \ref{thm:gweave_exists}.

\section{Non-uniqueness}\label{sec:unique}
Unlike the case of separable residue field extensions, the cyclic field $L/K$ constructed in Theorem \ref{thm:gweave_exists}, is not unique if $l/k$ is purely inseparable even when $K$ is complete and $[L:K]= [l:k]$. Assume that $K$ is complete discrete valued with uniformizer $t$.  We will show that, given any finite purely inseparable modular $l/k$  and any $l/k$-admissible gene $\mathcal{G}$ of length $[l:k]$ over $K$, one can construct infinitely many non-isomorphic $\mathcal{G}$-cyclic extensions.  For example, let $n=[l:k]$ and let $\mathcal{G} = g_1g_2\cdots g_n$ be an $l/k$-admissible gene. Since the length of $\mathcal{G}$ is equal to $[l:k]$, every $g_i$ is of ferocious type. Let  $g_1 = F_a$ for some $a \in \mathcal{O}_K$.  With the notations in Theorem \ref{thm:gweave_exists},   taking $\alpha_0 =0$,  the  Artin-Schreier extension over $K$ constructed in the first step of the induction is defined by 
    \begin{align} 
        x_1^p - x_1 = a/t^{p^{l_1}}
    \end{align}
  with $l_1$ satisfying the condition that  $l_1\geq N$. To emphasis the dependence of this extension on the choice of $l_1$, let us denote this extension by $L_1(l_1)$.   For a different choice $l_1'$ such that $l_1'\geq N$, we get  $L_1(l_1')$ defined by 
   \begin{align} \label{eqn:temp12}
        x_1^p - x_1 = a/t^{p^{l_1'}}
    \end{align}
    Both these extensions have the same residual extension, namely $k( \overline{a}^{1/p})$
     However,  since  $l_1\neq l_2$, $a/t^{p^{l_1}}   \notin c\cdot a/t^{p^{l_1'}} + \mathcal{P}(K)$ for every  $c \in \mathbb{F}_p$. Therefore,  by Artin-Schreier theory (\cite[Chapter VI, Theorem 8.3]{lang_algebra}), $L_1(l_1)$ and $L_1(l_1')$  are  non-isomorphic. In particular, the construction in Theorem \ref{thm:gweave_exists} yields different $\mathcal{G}$-cyclic extensions for the same $\mathcal{G}$. In fact, it is clear that there are infinitely many $\mathcal{G}$-cyclic extensions.

\section{Tower of $\mathcal{G}$-cyclic extensions over unramified and split cyclic extensions} \label{sec:tower}

Let  $F$ be a discrete valued field of characteristic $p$ and let  $t$ be a uniformizer. Recall the definition of \emph{split extensions} of $F$ from  \S\ref{sec:notn} (Definition  \ref{defn:split}). In general, $F$ need not necessarily  admit a split Artin-Schreier extension (for example, when $F$ is complete).  Suppose $F$ does admit a split  Artin-Schreier extension, we first show that  $F$ admits split Artin-Schreier-Witt  extensions  of any finite degree. Secondly, by the structure theorem of discrete valued Artin-Schreier-Witt  extensions (Theorem \ref{thm:structure}), if $K/F$ is a  completely ramified Artin-Schreier-Witt  extension, it is not possible to construct a unramified or split  Artin-Schreier extension $L/K$ such that $L/F$ is cyclic. However we may ask if we can construct a completely ramified  extension over any given  unramified or split Artin-Schreier-Witt  extension such that the resulting tower is cyclic. We prove  that this is indeed true in a more general setting.

\begin{lemma}\label{lem:splitASW}
    Suppose  $F$ admits a split Artin-Schreier extension.  Then  there exists a  split  Artin-Schreier-Witt  extension of  degree $p^m$  for any $m$.
    \end{lemma}
    \begin{proof}
    
       Let $L_1/K$ be a split Artin-Schreier extension.  By Lemma \ref{lem:valram}, 
       \begin{align*}
           L_1\simeq K[x_1]/(x_1^p - x_1 -a)
       \end{align*}
       for some $a \notin \mathcal{P}(K)$ with $a \in \mathfrak{m}_K$.  Let $W_m(K)$  denote the group of truncated Witt vectors of length $m$ and let $\mathcal{P}$ denote the Artin-Schreier operator on $W_m(K)$:
       \begin{align*}
           \mathcal{P}: W_m(K) &\rightarrow W_m(K)\\
            \mathbf{x} &\rightarrow \mathbf{x}^p - \mathbf{x}
       \end{align*}
    
Let   $L/K$ denote  the Artin-Schreier-Witt  extension corresponding to the Witt vector $\mathbf{a} = (a, 0,\cdots, 0) \in W_m(K)$ (\cite[Chapter VI, Exercise 50]{lang_algebra}). First note that  the extension is generated by $\mathbf{x} = (x_1,x_2, \cdots, x_m)$ satisfying
       \begin{align}\label{eqn:asw}
           \mathbf{x}^p - \mathbf{x} = \mathbf{a}
       \end{align} 
in the Witt vector arithmetic. Since $a \notin \mathcal{P}(K)$, the order of $\mathbf{a} \in W_m(K)/\mathcal{P}(W_m(K))$ is $p^m$. Therefore, $L/K$ is an Artin-Schreier-Witt  extension of degree $p^m$ (this also follows from \cite[Chapter VI, Exercise 50]{lang_algebra}). Let  
       \begin{align*}
       \mathbf{s} := (s_1, s_2, \dots s_m) = \mathbf{x} + \mathbf{a}
       \end{align*}
   We claim the following:\\
 \emph{Claim:} For every $i\leq m$
       \begin{align*}
           s_i \in  x_i + h_i
       \end{align*}
where $h_i \in a\mathbb{Z}[a,x_1, x_2, \cdots, x_{i-1}]$. \\
\indent Suppose the claim is true. Let $B_i$ is the integral closure of $\mathcal{O}_K$ in $L_i$ (the unique cyclic subextension of degree $p^i$ over $K$). Then we see  from (\ref{eqn:asw}) that the extension $L_i/L_{i-1}$ generated by $x_i$ satisfies 
\begin{align*}
    x_i ^p -x_i  = h_i \in \mathfrak{m}_K B_{i-1} \subseteq  \bigcap_{\mathfrak{p} \in Spec~ B_{i-1}} \mathfrak{p}
\end{align*}
By Lemma \ref{lem:valram},  this implies $L_i/L_{i-1}$ is split . Since tower of split extensions is split,  $L/K$ is split. So it suffices to prove the claim.\\
\indent We prove the claim by induction on $i$. Since $s_1 = x_1 + a$, we take $h_1= a$ and  the claim is clearly true for $i=1$. Assume that the claim is true for every $j\leq i-1$.  Now the following holds  in  $\mathbb{Z}[a, x_1, x_2, \cdots x_i]$(\cite[Chapter VI, Exercise 46, 47]{lang_algebra})
\begin{align*}
    s_1^{p^i}+ ps_2^{p^{i-1}}+ \cdots p^is_i &= x_1^{p^i}+ px_2^{p^{i-1}}+ \cdots + p^ix_i + a^{p^i}\\
     (x_1 + h_1)^{p^i}+ p(x_2 + h_2)^{p^{i-1}}+ \cdots + p^{i-1}(x_{i-1} + h_{i-1})^p + p^is_i &= x_1^{p^i}+ px_2^{p^{i-1}}+ \cdots +p^ix_i + a^{p^i}
\end{align*}
By induction hypothesis, $h_j \in a \mathbb{Z}[a, x_1, x_2, \cdots x_{j-1}]$.  Expanding  brackets, cancelling out terms on both the sides of the above equation and using the fact  that  $s_i \in \mathbb{Z}[a, x_1, x_2, \cdots x_i]$ (\cite[Chapter VI, Exercise 46]{lang_algebra}), we get 
\begin{align*}
    s_i -x_i \in a\mathbb{Z}[a, x_1, x_2, \cdots x_{i-1}]
\end{align*}
as claimed.
\end{proof}

\begin{remark} \normalfont \label{rmk:basechange}
    Suppose $F$ is a discrete valued field  and let $\mathcal{G}$ be a gene over $F$ with  a given  alphabet set. Suppose $K/F$ is either a split or an unramified   extension. Note that  for any discrete valuation on $K$ extending the one on  $F$,  $\Gamma_K = \Gamma_F$ and any uniformizer of $F$ is a uniformizer of $K$. Moreover, $k/f$ is either trivial or separable. Therefore, any subset of elements of $f$ that are $p$-independent over $f^p$ are also $p$-independent over $k^p$ (\cite[Chapter 2, Lemma 2.7.3]{fried_fa}). Hence, via base change, we see that $\mathcal{G}$ is also a gene over $K$ with  the same alphabet set.  We abuse notation and  denote  the resulting gene over $K$ also by $\mathcal{G}$.  
    \end{remark}
The following theorem is a slight variant of Theorem \ref{thm:gweave_exists}.
\begin{thm} \label{thm:gweave_tower}
     Let $\mathcal{G}$ be a gene of length $n$ over $F$ and let $N > \Sigma(n)$. Suppose $K/F$ is either a split or an unramified Artin-Schreier-Witt  extension. Let $u_{g_i} = z_{g_i}$ where $z_{g_i}$ is given by (\ref{eqn:initial}).  Then one can inductively construct a cyclic extension $L/K$ of degree $p^n$  with  the following properties: For every $1 \leq i \leq n$,
     \begin{enumerate}
     \item $L_i/F$ is cyclic where $L_i$ is the unique cyclic subextension in $L$ of degree $p^i$ over $K$
      \item    $L_i \simeq L_{i-1}[x_i]$  where $x_i$ satisfies
     \begin{align*}
         x_i^p - x_i = u_i + \alpha_{i-1}^{p^N}
     \end{align*}
     where $u_i = (1/t^{p^{l_i}}) u_{\mathfrak{a}(i)}$ for some $ l_i \geq N$ and $\alpha_{i-1}$ is any Albert element of  $L_{i-1}/F$.
     \item  there exists $z_i \in L_i$  such that $x_i \in L_{i-1}^{\{N-\Sigma(i), i\}}[z_i]$  and which satisfies 
     \begin{align*} 
         z_i^p - \gamma_{1i} z_i = z_{\mathfrak{a}(i)} + \gamma_{2i}
         \end{align*}
         for some $\gamma_{1i}, \gamma_{2i} \in \mathfrak{m}_{L_{i-1}}$.
\end{enumerate}
In particular, $L/K$ is a $\mathcal{G}$-weave and $L/F$ is cyclic.
\end{thm}
\begin{proof}
The construction of $L/K$ is essentially the same as the one in Theorem \ref{thm:main}. The only difference is that  in step (2),  we choose $\alpha_{i-1} \in L_{i-1}$  to be Albert element of  $L_{i-1}/F$  instead of $L_{i-1}/K$. By Theorem \ref{thm:albert} and Lemma \ref{lem:albert},  $L_i/F$ cyclic so that (1) holds.
\end{proof}

\begin{thm} \label{thm:splitunrammain}
    Let $K/F$ be a discrete valued finite Artin-Schreier-Witt  extension that is either  unramified or split.  Then given an  arbitrary gene $\mathcal{G}$ over $F$, one can construct a  $\mathcal{G}$-cyclic extension $L/K$  such that $L/F$ is cyclic .
\end{thm}
\begin{proof}
    This directly follows from Theorem \ref{thm:gweave_tower} and Theorem \ref{thm:gweave_gcyclic}.
\end{proof}
\begin{cor}
    Suppose $k$ admits an Artin-Schreier extension. Then given any $m\geq 0$ and an arbitrary gene $\mathcal{G}$ over $K$, one can construct Artin-Schreier-Witt  extension $L/K$ such that $L_m/K$ is unramified and $L/L_m$ is $\mathcal{G}$-cyclic. 
\end{cor}
\begin{proof}
    Since $k$ admits an Artin-Schreier extension, there exists  Artin-Schreier-Witt extensions  $l'/k$ of degree $p^m$ for any $m\geq 0$ by Theorem \ref{thm:albert}. Let $L'/K$ be an unramified lift of $l'/k$.  By Theorem \ref{thm:splitunrammain}, given any gene $\mathcal{G}$  over $K$ there exists  a $\mathcal{G}$- cyclic $L/L'$  extension such that $L/K$ is cyclic. It is easy to see that $L'=L_m$ and the rest of the claim follows. 
\end{proof}
\begin{cor}\label{cor:unrammain}
    Given any finite  extension $l/k$ such that $l \simeq l_1 \otimes_k l_2$ where $l_1/k$ is Artin-Schreier-Witt  and $l_2/k$ is  purely inseparable modular, there exists  a finite discrete valued Artin-Schreier-Witt  extension $L/K$ with residue $l/k$.
\end{cor}
\begin{proof}
    Take an unramified lift of $l_1/k$ and let $\mathcal{G}$ be an $l_2/k$-admissible gene over $K$. Apply Theorem \ref{thm:splitunrammain} and use the fact that $l_1$ and $l_2$  are linearly disjoint over $k$ (\cite[Chapter 2, Lemma 2.5.4]{fried_fa}).
\end{proof}
    \begin{cor} \label{cor:splitmain}
    Suppose $K$ admits a  split Artin-Schreier extension. Then given any $m\geq 0$ and an arbitrary gene $\mathcal{G}$ over $K$, one can construct Artin-Schreier-Witt  extension $L/K$ such that $L_m/K$ is split and $L/L_m$ is $\mathcal{G}$-cyclic (w.r.t any chosen discrete valuation on $L_m$). In particular, given a finite purely inseparable modular  $l/k$, an $l/k$-admissible gene $\mathcal{G}$ over $K$ and any $m \geq 0$, there exists  discrete valued Artin-Schreier-Witt  extensions $L/K$ with residue $l/k$ such that $L_m/K$ is split and $L/L_m$ is $\mathcal{G}$-cyclic.
\end{cor}
\begin{proof}
    This follows from Lemma \ref{lem:splitASW} and Theorem \ref{thm:splitunrammain}.
\end{proof}

\appendix
\section{Structure of   discrete valued  Artin-Schreier-Witt extensions} \label{app}
 Let $K$ be a discrete valued field of characteristic $p$ with valuation $v$ and value group $\Gamma_K = \mathbb{Z}$.  As before the lower case alphabets denote the residues of the corresponding upper case alphabets and $\mathcal{P}$ denotes the Artin-Schreier operator, $\mathcal{P}(x) = x^p-x$.  In this section we show that,
\begin{thm}\label{thm:structure}
   Let $L/K$ be a finite Artin-Schreier-Witt   extension   of discrete valued fields. Then $L/K$ is  of the form $L\supseteq L' \supseteq K$ where $L/L'$ is trivial or completely ramified and  $L'/K'$ is trivial or split or unramified.
\end{thm}

We start with a few lemmas.
\begin{lemma}\label{lem:trace}
    Let $L/K$ be a finite completely ramified Artin-Schreier-Witt extension. Then
    \begin{align*}
        Tr_{L/K}(\mathcal{O}_L) \subseteq \mathfrak{m}_K
    \end{align*}
\end{lemma}
\begin{proof}
   Let $[L:K] = p^n$. Since $L/K$ is completely ramified, any $\sigma \in Gal(L/K)$ acts trivially on $l/k$.  So for any $u \in \mathcal{O}_L$, 
    \begin{align*}
        \sigma(u) = u \text{~mod~} \mathfrak{m}_L
    \end{align*}
and
\begin{align*}
    Tr_{L/K}(u) \in  p^n u + \mathfrak{m}_K  = \mathfrak{m}_K
\end{align*}
\end{proof}

\begin{defn}
    Given $\alpha \in K$, we define
    \begin{align*}
        v_{\alpha} = \sup\{v(\alpha + \gamma^p-\gamma) | \gamma \in K\}
    \end{align*}
\end{defn}

\begin{remark}\label{rmk:optimal}\normalfont
     Suppose $v_{\alpha}\in \mathbb{Z}$, then there exists an element in $\beta \in \alpha + \mathcal{P}(\alpha)$, with $v_{\alpha} = v(\beta)$. If $v_{\alpha} = \infty$, then there exists an element in $\beta \in \alpha + \mathcal{P}(\alpha)$, with $v(\beta)>0$ i.e., $\beta \in \mathfrak{m}_K$. 
\end{remark}
\begin{defn}\normalfont
    An element $\alpha \in K$ is said to be \emph{optimal} if either $v(\alpha)>0$ or $v(\alpha) = v_{\alpha}$.
\end{defn}
By Remark \ref{rmk:optimal}, for any $\alpha \in K$, there is an optimal element in $\alpha + \mathcal{P}(K)$.
\begin{lemma}\label{lem:valram}
    Let $L/K$ be an  Artin-Schreier extension generated by $x$ satisfying
    \begin{align}\label{eqn:temp5}
        x^p - x = \alpha
    \end{align}
 where $\alpha$ is optimal. Then $L/K$ is 
    \begin{enumerate}[(i)]
  \item wildly ramified iff $v(\alpha) <0, (v(\alpha),p)=1$
    \item ferocious iff $v(\alpha)<0, (v(\alpha),p) \neq 1$
    \item unramified iff $v(\alpha) =0$
      \item split iff $v(\alpha) >0$
    \end{enumerate}
\end{lemma}
\begin{proof}
 Suppose $v(\alpha) <0, (v(\alpha),p)=1$. Then (\ref{eqn:temp5}) implies that $v(x) \in \frac{1}{p} \mathbb{Z} - \mathbb{Z}$ which makes $L/K$ wildly ramified.\\
\indent Suppose $v(\alpha) <0, (v(\alpha),p) \neq 1$. Let $t$ be a uniformizer of $\mathcal{O}_K$. Write $\alpha = u/t^{mp^r}, u \in \mathcal{O}_K^{\times}, r>0, (m,p)=1$.   Let $y= t^{mp^{r-1}}x$. Then (\ref{eqn:temp5}) implies 
\begin{align}\label{eqn:temp6}
    y^p - t^{mp^r - mp^{r-1}} y = u
\end{align}
Suppose $\overline{u} \in k^p$, then $u = c^p + m$ for some  $c\in \mathcal{O}_K, m\in\mathfrak{m}_K$ and $\alpha \in  c/t^{mp^{r-1}} + m/t^{mp^r} + \mathcal{P}(K)$ contradicting that $v(\alpha) <0$ and $\alpha$ is optimal. Therefore $\overline{u} \notin k^p$.
Taking residues of the equation (\ref{eqn:temp6}), we get $\overline{y}  = {\overline{u}}^{1/p}$. Therefore $L/K$ is ferocious. \\
\indent Suppose  $v(\alpha) =0$. Suppose $\overline{\alpha} = \overline{b}^p - \overline{b} \in \mathcal{P}(k)$ for some $b\in \mathcal{O}_K$. Then the element $\alpha  + b -b^p  \in \mathfrak{m}_K$ has valuation strictly greater than $v(\alpha)$ contradicting  the assumption that $v(\alpha) =0$ and  $\alpha$ is optimal. Therefore $\overline{\alpha} \notin P(k)$. Now taking residues of the equation (\ref{eqn:temp5}), we see that $l/k$ is an Artin-Schreier (hence separable) extension. Therefore we conclude that $L/K$ is unramified.
\\
\indent Suppose $v(\alpha) >0 $. Let $\hat{K}$ denote the completion of $K$. By Cohen's structure theorem (\cite[Theorem 15]{cohen}), $\hat{K} \simeq k((t))$ is the field of Laurent series over $k$ and $\alpha$ can be identified with an element of the form $\sum_{i \geq 1}a_it^i$ for some $a_i \in k$. Now $(\alpha + \alpha^p  + \alpha^{p^2} \cdots ) \in \hat{K}$ and we have 
    \begin{align*}
        \alpha = (\alpha + \alpha^p  + \alpha^{p^2} \cdots )- (\alpha + \alpha^p  +\alpha^{p^2} \cdots )^p  \in \mathcal{P}(\hat{K})
    \end{align*}
    Therefore $L \otimes_K \hat{K} \simeq  \oplus \hat {K}$ (one can also directly show this via Hensel's lemma) and hence $L/K$ is split (\cite[Chapter II, Proposition 8.3]{neukirch}).\\

\end{proof}

  \begin{lemma} \label{lem:comram} Let $L/K$ be a  finite Artin-Schreier-Witt  extension that is completely ramified. Then every  Artin-Schreier extension $M/L$ such that $M/K$ is cyclic is completely ramified. 
    \end{lemma} 
    \begin{proof}
    To show that $M/L$ is completely ramified, it suffices to show that  $M/L$ is not split  or unramified. First note that by Remark \ref{rmk:dvr}, the integral closure of $\mathcal{O}_K$ in $L$ is a discrete valuation ring.    Suppose $M/L$ is split or unramified. Let $\sigma \in Gal(L/K)$ be a generator.  By  Theorem \ref{thm:albert},   $M/L$  is generated by $x$ satisfying
        \begin{align*}
            x^p - x = \alpha
        \end{align*}
        for some $\alpha \in L$ such that 
        \begin{align}\label{eqn:beta}
            \sigma(\alpha) - \alpha = \beta^p - \beta
        \end{align}
        for some $\beta \in L/K$ with $Tr_{L/K}(\beta) = 1$. Note that we may replace $\alpha$ by $\alpha' = \alpha + \gamma^p - \gamma$ and $\beta$ by $\beta' = \beta + \sigma(\gamma) - \gamma$ so that $Tr_{L/K}(\beta') = 1$. So  we may assume that $\alpha$ is optimal. Since $M/L$ is split or unramified, by Lemma \ref{lem:valram} we have $\alpha \in \mathcal{O}_L$.    But (\ref{eqn:beta}) implies that $\beta \in \mathcal{O}_L$. By Lemma \ref{lem:trace}, this contradicts the fact that $Tr_{L/K}(\beta) = 1$.
\end{proof}

\begin{lemma}\label{lem:splitunram}
    Let $L/K$ be a  finite discrete valued Artin-Schreier-Witt  extension that is unramified (resp. split). Let $M/L$ be 
 an Artin-Schreier extension that  is  split (resp. unramified w.r.t to a chosen discrete valuation on $L$ extending the one on $K$). Then $M/K$ is not cyclic.
\end{lemma}
\begin{proof}
    Suppose $L/K$ is split and equipped with  a chosen discrete valuation on $L$ extending the one on $K$. Note that $l=k$.  Suppose $M$ is unramified over $L$ with residue $m$. Then  $M/K$ contains a lift $M'/K$ of $m/k$. But $[m:k] =p$. Therefore $M'/K$ is cyclic of degree $p$ contained in the  extension $M/K$. If $M/K$ is cyclic, $M' \subseteq L$ which contradicts the fact that $L/K$ is split. Therefore, $M/K$ is not cyclic.\\
    \indent Now assume that  $L/K$ is unramified and $M/L$ is split.  By Remark \ref{rmk:dvr}, the integral closure of $\mathcal{O}_K$ in $L$ is a discrete valuation ring. Suppose  $M/K$ is cyclic. Let $\sigma \in Gal(L/K)$ be a generator. Then by Theorem \ref{thm:albert}, $M/L$ is generated by $x$ satisfying
    \begin{align*}
        x^p - x = \alpha
    \end{align*}
    where $\sigma(\alpha) - \alpha = \beta^p -\beta$ for some $\beta \in L$ with $Tr_{M/L}(\beta) =1$. As before, by replacing $\alpha$ by an element in the coset $\alpha +\mathcal{P}(L)$, we may assume that $\alpha$ is optimal. Hence,  by Lemma \ref{lem:valram},  $\alpha \in \mathfrak{m}_L$ and therefore   $\sigma(\alpha) - \alpha \in \mathfrak{m}_L$. This implies that $\beta \in \mathcal{O}_L$ and its residue $\overline{\beta} = c$, for some $c \in \mathbb{F}_p$.  Therefore $\beta \in c + \mathfrak{m}_L$ and 
    \begin{align*}
        Tr_{L/K}(\beta) \in  c[L:K] +  \mathfrak{m}_L = \mathfrak{m}_L
    \end{align*}
    This contradicts the fact that $Tr_{L/K}(\beta) =1$. 
\end{proof}

\begin{proof}[Proof of Theorem \ref{thm:structure}] 
  This follows from Lemma \ref{lem:comram} and Lemma \ref{lem:splitunram}.
\end{proof}

\section{Auxiliary Lemmas} \label{sec:auxi}
\noindent In this section $F$ denotes a discrete valued field of characteristic $p$ with valuation $v$. Recall the binary relations $\prec$, $\preceq$ and $\approx$ from Definition \ref{defn:valcompare}.
\begin{lemma}\label{lem:li}
    Let $\phi \in F$ with $\phi\prec 1$. Then for any given  finite set $\{b,d_i,e_i |1\leq i \leq n\} \subseteq F$ and integers  $m_i \geq 1, s\geq 0$, there exists   $l \in \mathbb{Z}$ such that for every $i$
    \begin{align}\label{eqn:choosingl}
        \phi^{p^{l-s}}b \prec \phi^{p^{l-s-m_i}}d_i + e_i 
    \end{align}
\end{lemma}
\begin{proof}
       For each $i$ there are two possibilities, namely $d_i =0$  and $d_i \neq 0$. Suppose $d_i = 0$. In this case, clearly there exists $l_i$ such that $\phi^{p^{l_i-s}} b\prec e_i$ since $\phi \prec 1$.  Suppose $d_i \neq 0$. Then again  since $\phi \prec 1$, one can choose $n_i \in \mathbb{Z}$ such that $\phi^{p^{n_i-s-m_i}}d_i \prec e_i$.  Let $z$ be a variable.  Note that $\phi^{p^{z-s}}b \prec  \phi^{p^{z-s-m_i}} d_i $  holds if and only if 
    \begin{align*}
        p^{z-s} v(\phi) + v(b) <(p^{z-s-m_i})v(\phi) + v(d_i)\\
        \iff p^{z-s-m_i}(p^{m_i}-1) v(\phi) < v(d_i) -v(b)
    \end{align*}
    Since the right hand side of the above inequality is fixed and $\phi \prec 1$ there exists   $l_i  
    \geq n_i$ such that $z = l_i$ satisfies the above inequality.   It is easy to see that $l= \max\{l_i\}$ satisfies (\ref{eqn:choosingl}).
    
\end{proof}

\begin{lemma}\label{lem:simple}
    Let $j$ be a non-negative integer and let  $F^{alg}$ denote the algebraic closure of $F$. Given $a,b \in F^{alg}$, suppose the following holds in $F^{alg}[X,Z]$:  
  \begin{align*}
      X^p  - X  =  aZ^{p^j} + b  
  \end{align*}
 Then  for any $m \in \mathbb{N}$,
  \begin{align*}
      W =  X - \sum_{k=0}^{j-1} (a^{1/p^{j-k}})Z^{p^k} -\sum_{i=1}^{m} b^{1/p^i}
  \end{align*}
  satisfies 
  \begin{align*}
  W^p - W = a^{1/p^j}Z+b^{1/p^m}    
  \end{align*}
\end{lemma}
\begin{proof}
    This  is clear.
\end{proof}
Recall the definition of $\{m,n\}$-special polynomials  from Definition \ref{defn:mnspecial}.
\begin{lemma}\label{lem:simplify2}
Let $\alpha_1,\alpha_2 \in F$ and let $m > n$ be positive integers.   Suppose we are given integers  $l > m,N > m$ and $f(Z) \in F^{\{m, n\}}[Z]$. Assume that $X$ satisfies
\begin{align*}
    X^p - X = \frac{1}{t^{p^l}}(f(Z)^p -f(Z) - \alpha_1^{p^N}) + \alpha_2^{p^N}
\end{align*} 
Then there exists $g(Z) \in F^{\{m-n-1, n\}}[Z]$ such that $Y := X - g(Z)$ satisfies
\begin{align} \label{eqn:temp8}
    Y^p - Y = h Z + \eta
\end{align}
 for some $h \in F^{p^{m-n}}, \eta \in F^{p^{m-n-1}}$. Moreover, when  $l$ is large enough, $h \prec \{1, \eta\}$.
\end{lemma}
\begin{proof}
    Write
    \begin{align*}
        f(Z) = c + \sum_{j=s}^n b_jZ^{p^j}
    \end{align*}
where $c, b_j \in F^{p^m}, b_s \neq 0$ for some $s\geq 0$. Note that $f(Z)^p \in F^{\{m+1, n+1\}}[Z]$. By recursive application of Lemma \ref{lem:simple}, we conclude that that there exists $g(Z) \in F^{\{m-n-1, n\}}[Z]$ such that $Y = X- g(Z)$ satisfies
\begin{align*}
     Y^p - Y = (b_s^{p^{-s}}/t^{p^{l-s}}  + \sum_{j=s+1}^{n+1} d_j/t^{p^{l-j}}) Z + (1/t^{p^{l-s-1}}) (c^{p^{-s}} -c^{p^{-s-1}} - \alpha_1^{p^{N-s-1}}) +\alpha_2^{p^N}
\end{align*}
where $d_j \in F^{p^{m-n}}$  is a constant depending only on  $b_s, \cdots, b_n$. Let $h=(b_s^{p^{-s}}/t^{p^{l-s}}  + \sum_{j=s+1}^{n+1}d_j/t^{p^{l-j}})$ and $\eta = (1/t^{p^{l-s-1}}) (c^{p^{-s}} -c^{p^{-s-1}} - \alpha_1^{p^{N-s-1}}) +\alpha_2^{p^N}$ so that  (\ref{eqn:temp8}) holds. By Lemma \ref{lem:li}, there exists large enough $l$ for which $b_s^{p^{-s}}/t^{p^{l-s}} \prec \{ 1, \eta, d_j/t^{p^{l-j}}|s+1\leq j \leq n+1  \}$. Therefore $h \approx b_s^{p^{-s}}/t^{p^{l-s}}$ and   $h \prec \{1, \eta\}$ as required.
\end{proof}

\section*{Acknowledgements}
The author acknowledges the support of the DAE, Government of India, under Project Identification No. RTI4001. She thanks Najmuddin Fakhruddin for corrections and many helpful suggestions. 
\nocite*{}
\bibliographystyle{alpha}
\bibliography{ref}

\end{document}